\documentclass[12pt]{article} 
\usepackage[english]{babel}

\usepackage[T2A]{fontenc} \usepackage[english]{babel}

\usepackage[english]{babel}
\usepackage{amsthm, amsmath, amssymb, amsbsy, amscd, amsfonts, latexsym, euscript,epsf,stackrel}
\usepackage{authblk}
\usepackage{graphicx} 

\usepackage{placeins}
\usepackage{graphicx}
\usepackage{epsfig}
\newtheorem{theorem}{Theorem}[section]
\newtheorem{theor}[theorem]{Theorem}
\newtheorem{corollary}[theorem]{Corollary}
\newtheorem{lemma}[theorem]{Lemma}
\newtheorem{proposition}[theorem]{Proposition}

\newtheorem{definition}[theorem]{Definition}
\newtheorem{example}[theorem]{Example}
\newtheorem{remark}[theorem]{Remark}

\usepackage[all,cmtip]{xy}
\usepackage{xcolor}
\usepackage{xfrac}
\usepackage{pb-diagram}

\usepackage{tikz}
\usetikzlibrary{positioning,arrows}

\newcommand\vtr[2]{{
		(#1_1, \ldots, #1_{#2})
}}

\def\bea{\begin{eqnarray}}
	\def\eea{\end{eqnarray}}

\title{
	Embedding a graded matrix algebra into an elementary graded
}
\author{Pavel Sokolov}

\begin{document}
	
	\maketitle
	\begin{abstract}

		M.V.\,Zaicev and S.K.\,Segal, as well as S.\,Dăscălescu, B.\,Ion, C.\,Năstăsescu, and D.\,Raios Montes studied certain gradings on matrix rings and algebras - ‘elementary’ gradings. However, examples of gradings on a matrix ring that are not elementary are known. 
		In the present article, we show that any subalgebra of a full matrix ring over an arbitrary graded division ring can be embedded in an elementary graded matrix ring.

		
	\end{abstract}
	
	\section{Introduction}
	
	Graded rings and algebras are widely used in geometry and topology. For example, in homology and cohomology theories, graded abelian groups of chains, boundaries, and homologies arise naturally, and are endowed with the structure of a graded module over a graded algebra \cite{JWick}. Also, in works \cite{Tu2grp, TuHom3}, a certain class of graded algebras is considered to study two-dimensional homotopical quantum field theory.
	
	As a natural example of a graded algebra, one can consider the group algebra $A = \mathbb F[G]$ over a field $\mathbb F$. In this case, the homogeneous components are given as one-dimensional spans of elements of the group $G$, i.e., $A_g = \mathbb{F}g$. 
	Another example of a graded algebra is the algebra of polynomials $R = K[x_1, \ldots, x_n]$, where a $\mathbb Z$-grading is naturally introduced as follows: $R_n$ is the linear span of all monomials of degree $n$ at $n \geqslant 0$ and $R_n = 0$ otherwise. Obviously, in this case, $R_nR_m \subseteq R_{n+m}$ and $R = \bigoplus_{n\in \mathbb{Z}} R_n$.
	
	In the previous example, one can consider a grading not by the group $(\mathbb Z, +)$, but by the monoid $(\mathbb N, +)$. In this case, an $\mathbb N$-graded algebra can be considered as a filtered algebra as follows. Let $F_n$ be the vector space spanned by monomials of degree no greater than $n$, i.e., $F_n = \bigoplus_{i = 0}^nR_n$, then, $$\{0\}\subseteq F_0 \subseteq F_1 \subseteq \ldots \subseteq F_m \subseteq \ldots \subseteq K[x_1, \ldots, x_n].$$
	Obviously, the following properties are satisfied:  $\displaystyle K[x_1, \ldots, x_n] = \bigcup_{i\in \mathbb{N}} F_i$, $F_kF_t \subseteq F_{k+t}$.
	
	However, $\mathbb{Z}$-graded algebras can be generalized not only to filtered algebras, but also to algebras graded by an arbitrary group $G$. A group algebra with the natural grading is one of such examples.
	
	Let us consider one more example of graded algebras. Let $G$ be a group, $V$ be a vector space of dimension $n$ over a division ring $\mathbb D$. The algebra of endomorphisms of this space is $M_n(\mathbb D)$. One can consider the vector space $V$ as a graded $M_n(\mathbb D)$-module. It is sufficient to fix a basis $e_1, \ldots, e_n$ of this space and arbitrary elements $g_1 , \ldots , g_n$ of the group $G$. The homogeneous components of the space V are defined as follows: $V_{g_i} = span \{e_i\}$. Then, the algebra $M_n(\mathbb D)$ naturally acts on $V$ by the following rule: $E_{ij}V_{g_j} \subseteq V_{g_i}$, where $E_{ij}$ are the matrix units. Thus, the grading arises on the algebra $M_n(\mathbb D)$, such that $E_{ij} \in (M_n(\mathbb D))_{g_ig_j^{-1}}$. This grading is called elementary.
	
	Elementary gradings were studied in works $\cite{RiosMontes, BahSeg}$. Moreover, it was shown in these works that any simple Artinian ring, graded by a torsion-free group, is elementary graded.
	The authors of $\cite{BahSeg}$ gave an example of an $M_2(\mathbb{F})$ grading, where $\mathbb{F}$ is a field of characteristic not $2$, which is not elementary.
	
	The aim of this work is to research arbitrary gradings on matrix algebras and its relations with elementary gradings. This article is structured as follows. In Section \ref{Prel}, the main definitions and known results are given. In Section \ref{baseProps}, some simple properties of elementary gradings on $M_n(\mathbb D)$ are studied, which will be needed later. In Section \ref{subsec:essentialBasis}, we consider graded subalgebras of the algebra $M_n(\mathbb D)$ that have a basis consisting of homogeneous matrix units and their connection with elementary gradings. In Section \ref{sect:commonCase}, it is shown that any subalgebra of $M_n(\mathbb D)$ graded in an arbitrary way can be embedded into an elementary graded one.

	\section{Preliminaries} \label{Prel}
	
	We recall some basic definitions in this section. To begin with, let us remember the definition of a graded algebra. Consider an arbitrary group $G$, a ring $K$, and an algebra $L$ over $K$. The algebra $L$ is called $G$-graded, or briefly graded, if it is representable as a $K$-module in the form of a direct sum
	
	\[L = \bigoplus_{g\in G}L_g\]
	where all $L_g$ are $K$-modules and $L_gL_h \subseteq L_{gh}$ for all $g, h \in G$.
	In that case the modules $L_g$ will be called homogeneous components, and the elements $a \in L_g$ homogeneous of degree g, or briefly homogeneous elements.
	Consider two $G$-graded algebras $A = \displaystyle\bigoplus_{g\in G}A_g$, $B = \displaystyle\bigoplus_{g \in G}B_g$. A homomorphism of algebras $\varphi\colon A \to B$ is called a homomorphism of $G$-graded algebras if $\varphi(A_g) \subseteq B_g$ for all $g \in G$.
	Let us fix the following notation: the symbol $\mathbb D$ will denote an arbitrary division ring. A subalgebra $L \leqslant M_n(\mathbb D)$ of the full matrix algebra will be called a matrix algebra.
	
	Let us now recall the construction of an elementary graded algebra considered in work \cite{BahSeg}. Let $G$ be a group, $L = M_n(\mathbb D)$ be the algebra of square matrices of size $n \times n$ over the division ring $\mathbb D$. Fix a tuple $\overline g = (g_1, \ldots, g_n) \in G^n$. One can easily define the set $L_h$ for each $h \in G$ as follows:
	
	\[L_h = span\{E_{ij} \ | \ g_ig_j^{-1} = h\}.\]
	Then each set $L_h$ will be a $\mathbb D$-module, and the algebra $L$ is representable as a $\mathbb D$-module in the form $L = \bigoplus_{h \in G}L_h$, moreover $L_{h_1}L_{h_2} \subseteq L_{h_1h_2}$; i.e., the algebra L is G-graded. Such a grading is called elementary, and the algebra $L$ is elementary graded.

	\section{Properties of elementary gradings} \label{baseProps}

	In this section, we establish some properties of elementary gradings which will be needed later.
	
	\begin{proposition}\label{prop:shift}
		
		Let $L = M_n(\mathbb D)$ be an elementary $G$-graded algebra with respect to the tuple of elements $\overline{g} = \vtr{g}{n}$. Then, for each element $h \in G$, the elementary graded algebra $L$ with respect to the tuple $\overline{g}\cdot h = (g_1h,\ldots, g_nh)$ is such that $R_s = L_s$ for all $s \in G$.  
		
	\end{proposition}
	
	\begin{proof}
		
		Let $E_{ij} \in L_s$. By the definition of an elementary grading, $L_s = L_{(g_ih)(g_jh)^{-1}}$. Then
		
		$$s = (g_ih)(g_jh)^{-1} = g_ihh^{-1}g_j^{-1} = g_ig_j = s.$$
		Consequently, $E_{ij} \in R_{g_ig_j^{-1}} = R_s$. Thus, $L_s \subset R_s$. The reverse inclusion is proved in a similar way.
	\end{proof}
	
	\begin{corollary}
		
		Let $L = M_n(\mathbb D)$ be an elementary $G$-graded algebra with respect to the tuple of elements $\overline{g} = \vtr{g}{n}$, then there exists a tuple of elements $\overline{h} = \vtr{h}{n} \in G^n$ such that $h_1 = 1_G$, and the elementary gradings of the algebra $L$ with respect to the tuples $\overline{g}$ и $\overline{h}$ and are coincide.
	\end{corollary}
	
	In order to define an elementary grading on a matrix algebra, it is necessary to choose $n$ elements. The latter corollary shows that one can choose only $n-1$ elements and choose the grading so that $g_1 = 1$.
	
	\begin{remark}
		
		Let $L = M_n(\mathbb D)$ be an elementary $G$-graded algebra, then $E_{ii} \in L_1$ for all $i = 1, \ldots, n$. This follows directly from the definition, since $E_{ii} \in L_{g_ig_i^{-1}} = L_1$.
	\end{remark}

	The latter remark allows constructing a series of graded algebras which are not elementary graded. For example, if one requires the neutral homogeneous component to lie in the center of the algebra, i.e., $L_1 \subset Z(L)$, the grading on $L$ cannot be elementary. Such algebras are considered, for example, in \cite{Tu2grp}. Also, an algebra with this property is given as a counterexample showing that not all gradings are elementary in \cite{BahSeg}.

	\begin{corollary}\label{coro:badCenter}
		
		Let $L = M_n(\mathbb{D})$ be a $G$-graded algebra, with $L_1 \subseteq Z(L)$. Then $L$ is not an elementary graded algebra at $n \geqslant 2$.
	\end{corollary}

	\begin{proof}
		
		By the previous remark, $E_{ii} \in L_1$. It is obvious that $E_{ii} \not \in Z(L)$, which contradicts the assumption $L_1 \subseteq Z(L)$.
	\end{proof}

	\section{Algebras with a natural basis}\label{subsec:essentialBasis}
	
	In this section, we consider matrix algebras of a special form. Namely, graded matrix algebras that have a basis consisting of homogeneous matrix units. 
	Let us give a precise definition: a $G$-graded subalgebra $L \leqslant  M_n(\mathbb D)$ will be called a $G$-graded algebra with a natural basis, or simply an algebra with a natural basis, if the following conditions are holds:
	\begin{enumerate}
		
		\item $L = span\{E_{ij} \ | \ \text{for some } i, j\}$;
		\item Each $E_{ij} \in L$ are homogeneous.
	\end{enumerate}
	
	\begin{example}\label{example:upperTriangular}
		
		Consider the algebra of upper triangular matrices $L = T_n(\mathbb D) = span\{E_{ij} \ | \ 1 \leqslant i \leqslant j \leqslant n\}$ with a $\mathbb Z$-grading given by the following rule: 
		
		\[L_h = span \{E_{ij} \ | \ j - i = h, \ 1 \leqslant i \leqslant j \leqslant n\}.\]
		It is obvious that $L$ is the linear span of some matrix units, with all matrix units in $L$ being homogeneous, i.e. $L$ is a $\mathbb Z$-graded algebra with a natural basis. 
	\end{example}
	
	In this section, the symbol $L$ denotes a $G$-graded algebra with a natural basis. For each algebra $L$ with a natural basis we define two directed multigraphs $\gamma(L)$ and $\Gamma(L)$, such that $\gamma(L) \subseteq \Gamma(L)$. Let us begin by defining $\gamma(L)$. The set of vertices of $\gamma(L)$ is
	\[V(\gamma) = \{v_{ij} \ | \ E_{ij} \in L\}\]
	and the edges of this graph are defined as follows:
	\[E(\gamma) = \{(v_{ij}, v_{ik}) \ | \  v_{jk} \in V(\gamma)\} \cup \{(v_{ki}, v_{ji}) \ | \ v_{jk} \in V(\gamma)\}.\]
	Edges of the form $(v_{ij}, v_{ik})$ will be denoted by $e_{jk}^i$, and edges $(v_{ki}, v_{ji})$ will be denoted by $e^{jk}_i$ for brief. As a result, the graph $\gamma(L) = (V(\gamma), E(\gamma))$. Edges of the form $e_{jk}^i$ represents the multiplication of the matrix unit $E_{ij}$ by $E_{jk}$ on the right, since multiplying $E_{ij}$ on the right by $E_{jk}$ gives $E_{ik}$. The same analogy for multiplication on the left is also valid for edges of the form $e_i^{jk}$.
	
	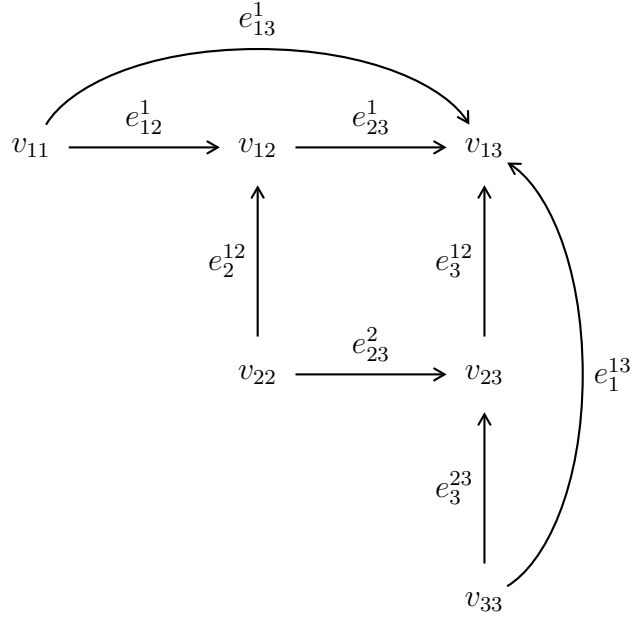
\begin{figure}
		\centering
		\begin{tikzpicture}
			\node at (0, 6) {$v_{11}$};
			\node at (3, 6) {$v_{12}$};
			\node at (6, 6) {$v_{13}$};
			\node at (3, 3) {$v_{22}$};
			\node at (6, 3) {$v_{23}$};
			\node at (6, 0) {$v_{33}$};

			\node at (3.0, 7.7) {$e_{13}^1$};
			\draw[thick, -angle 60] 
			(0.2, 6.3) 
			.. controls (0.5, 6.8) and (1.5, 7.3) .. 
			(3.0, 7.3) 
			.. controls (4.5, 7.3) and (5.5, 6.8) .. 
			(5.8, 6.3);

			\node at (7.7, 3.0) {$e^{13}_1$};
			\draw[thick, -angle 60] 
			(6.3 , 0.2) 
			.. controls (6.8, 0.5) and (7.3, 1.5) .. 
			(7.3, 3.0) 
			.. controls (7.3, 4.5) and (6.8, 5.5) .. 
			(6.3, 5.8);

			\node at (1.5, 6.4) {$e_{12}^1$};
			\draw[thick, -angle 60] (0.5, 6.0) -- (2.5, 6.0);
			\node at (4.5, 6.4) {$e_{23}^1$};
			\draw[thick, -angle 60] (3.5, 6.0) -- (5.5, 6.0);
			\node at (4.5, 3.4) {$e_{23}^2$};
			\draw[thick, -angle 60] (3.5, 3.0) -- (5.5, 3.0);
			\node at (5.6, 1.5) {$e^{23}_3$};
			\draw[thick, -angle 60] (6.0, 0.5) -- (6.0, 2.5);
			\node at (5.6, 4.5) {$e^{12}_3$};
			\draw[thick, -angle 60] (6.0, 3.5) -- (6.0, 5.5);
			\node at (2.6, 4.5) {$e^{12}_2$};
			\draw[thick, -angle 60] (3.0, 3.5) -- (3.0, 5.5);

		\end{tikzpicture}
		\caption{Example $\gamma(T_3(\mathbb{D}))$.} \label{fig:example}
	\end{figure}
	
	Consider the set $V^T(\gamma) = \{v_{ji} \ | \ v_{ij} \in V(\gamma)\}$. Define the graph $\Gamma(L)$ the vertices of which are
	
	\[V(\Gamma) = \{v_{ij} \ | \ \exists v_{ik_1}, v_{k_1k_2}, \ldots v_{k_nj} \in V(\gamma)\cup V^T(\gamma)\}\]
	and the edges of this graph are
	\[E(\Gamma) = \{(v_{ij}, v_{ik}) \ | \  v_{jk} \in V(\Gamma)\} \cup \{(v_{ki}, v_{ji}) \ | \ v_{jk} \in V(\Gamma)\}.\]
	Using the analogy with the notation of the graph $\gamma(L)$, edges of the form $(v_{ij}, v_{ik})$ will be denoted by $e_{jk}^i$, and edges $(v_{ki}, v_{ji})$ will be denoted by $e^{jk}_i$.
	
	The edges and vertices are denoted by the same symbols in the graphs $\gamma(L)$ and $\Gamma(L)$. Further, it will be clear from the context which graph a given edge, or vertex belongs to. The example of the graph $\gamma(T_3(\mathbb D))$, where the grading on the algebra $T_3(\mathbb D)$ is given in the same way as in example \ref{example:upperTriangular}, is illustrated in the Figure \ref{fig:example}. The loops in the Figure \ref{fig:example} were omitted to obtain a compact image. However, it should be noted that each vertex $v_{ij}$ has two loops: $e^i_{jj}$ and $e_j^{ii}$ in the graph $\gamma(T_3(\mathbb D))$.
	
	Further, we will need some properties of the graphs $\gamma(L)$ and $\Gamma(L)$.

	\begin{proposition}\label{prop:graphProps}
		The following properties of the graphs holds:
		
		\begin{enumerate}

			\item $\gamma(L) \subseteq \Gamma(L)$;
			\item If $v_{ij} \in V(\Gamma)$, then $v_{ji} \in V(\Gamma)$;
			\item If $v_{ij}, v_{ik}\in V(\Gamma)$, then they lie in the same connected component of the graph $\Gamma(L)$;
			\item If$v_{ji}, v_{ki} \in V(\Gamma)$, then they lie in the same connected component of the graph $\Gamma(L)$;
			\item Each connected component of the graph $\Gamma(L)$ contains at least one vertex of the graph $\gamma(L)$.
		\end{enumerate}
	\end{proposition}
	
	\begin{proof}
		1) It is clear.
		
		\noindent 2) Let $v_{ij} \in \Gamma(L)$, then there exists a sequence of vertices
		\[v_{ik_1}, v_{k_1k_2}, \ldots, v_{k_nj} \in V(\gamma) \cup V^T(\gamma).\]
		Hence, there exists another sequence of vertices \[v_{jk_n}, v_{k_nk_{n-1}}, \ldots, v_{k_1i} \in V(\gamma) \cup V^T(\gamma),\]
		and therefore, $v_{ji} \in V(\Gamma)$. 
		
		\noindent 3) Let $v_{ij}, v_{ik} \in V(\Gamma)$, then by the previous item $v_{ji} \in V(\Gamma)$. Consequently, there exists a vertex $v_{jk} \in V(\Gamma)$, which means that there exists an edge $e_{jk}^i \in E(\Gamma)$, directed from $v_{ij}$ to $v_{ik}$ as well.
		
		\noindent 4) Similarly, there exists the edge $e^{jk}_i \in E(\Gamma)$ from $v_{ki}$ to $v_{ji}$.
		
		\noindent 5) Suppose that there is no vertex from $V(\gamma)$ in the connected component $\mathcal{O} \subseteq \Gamma(L)$, i.e., $\mathcal{O}\cap V(\gamma) = \emptyset$, then we prove that $\mathcal{O}\cap V^T(\gamma) = \emptyset$. Indeed, let $v_{ij} \in \mathcal{O}\cap V^T(\gamma)$, then by item $2$ $v_{ji} \in V(\Gamma)$. Consequently, there are vertices $v_{ij}$, $v_{ii}$ connected by the edge $e_{ji}^i$. Hence, $v_{ii} \in \mathcal{O}$. Since $v_{ij}$, $v_{ii} \in \mathcal{O}$, then $v_{ji} \in \mathcal{O}$ as well, since $v_{ji}$ is connected with $v_{ii}$ by the edge $e^{ij}_i$. It contradicts the fact that $v_{ji} \in V(\gamma)$.
		
		Now let $v_{ij} \in \mathcal{O}$ be an arbitrary vertex, then there exists a sequence of vertices $v_{ik_1}, v_{k_1k_2}, \ldots v_{k_nj} \in V(\gamma)\cup V^T(\gamma) \subseteq V(\Gamma)$. Without losing any generality, we may assume that the number of vertices in this sequence is greater than $1$, since, if the sequence consists of one vertex $v_{st} \in V(\gamma)\cup V^T(\gamma)$ then $v_{ts} \in V(\gamma)\cup V^T(\gamma)$, and one can consider the sequence $v_{st},v_{ts},v_{st}$ instead of $v_{st}$.
		Since $v_{ik_1}, v_{k_1k_2}, \ldots v_{k_nj} \in V(\gamma)\cup V^T(\gamma)$, then $v_{jk_n}, v_{k_nk_{n-1}}, \ldots, v_{k_2k_1} \in V(\gamma) \cup V^T(\gamma)$. It states that there are vertices $v_{jk_n}, v_{ik_1} \in V(\Gamma)$, and hence there is an edge $e_{jk_n}^i$ from $v_{ij}$ to $v_{ik_n}$. Consequently, $v_{ik_n} \in \mathcal{O}$, but  $v_{ik_n} \in V(\gamma) \cup V^T(\gamma)$. There is a contradiction.
	\end{proof}
	
	Now, we define a map that assigns elements of the group $G$ to the vertices and edges of the graphs $\gamma(L)$ and $\Gamma(L)$. Consider the graph $\gamma(L) = (V(\gamma), E(\gamma))$ and define $\varphi\colon \gamma(L) \to G$ by the following rule:
	
	\begin{enumerate}
		
		\item Let $v_{ij} \in V(\gamma)$, then $\varphi(v_{ij}) = g$, where $E_{ij} \in L_g$.
		
		\item Let $e_{jk}^i \in E(\gamma)$, then $E_{jk} \in L$ and $v_{jk} \in V(\gamma)$. Define $\varphi(e_{jk}^i) = \varphi(v_{jk})$. Similarly, $\varphi(e^{jk}_i) = \varphi(v_{jk})$ for the edges $e^{jk}_i \in E(\gamma)$.
	\end{enumerate}
	
	Note that the $\varphi$ is well defined, since $L$ is an algebra with a natural basis; in particular, all matrix units in $L$ are homogeneous. Let us establish some simple properties of this map.

	\begin{lemma}\label{lemma:phiProperties}
		
		The following properties holds:
		
		\begin{enumerate}
			
			\item If $v_{ii} \in V(\gamma)$, then $\varphi(v_{ii}) = 1 \in G$;
			\item If $v_{ij}, v_{jk} \in V(\gamma)$, then $\varphi(v_{ij}) \cdot \varphi(v_{jk}) = \varphi(v_{ik})$;
			\item If $v_{ij}, v_{ji} \in V(\gamma)$, then $\varphi(v_{ij}) = \varphi(v_{ji})^{-1}$.
		\end{enumerate}
		
	\end{lemma}

	\begin{proof}
		
		1) Let $v_{ii} \in V(\gamma)$, then $E_{ii} \in L$; since all matrix units in $L$ are homogeneous, $E_{ii} \in L_g$ for some $g \in G$. It has to be proved that $g = 1$. Since a grading is given on the algebra $L$, we have $L_g \ni E_{ii} = E_{ii}E_{ii} \in L_{g^2}$. Consequently, $g = g^2$ and $g = 1$.
		
		\noindent 2) Let $v_{ij}, v_{jk} \in V(\gamma)$, then $E_{ij} \in L_g$, $E_{jk} \in L_h$ for some $g, h \in G$ and $E_{ik} = E_{ij}E_{jk} \in L_{gh}$. Consequently, $E_{ik} \in L_{gh}$ and $\varphi(v_{ik}) = gh = \varphi(v_{ij})\cdot \varphi(v_{jk})$.
		
		\noindent 3) Let $v_{ij}, v_{ji} \in V(\gamma)$, then by item $2$ we obtain that $\varphi(v_{ij})\cdot\varphi(v_{ji}) = \varphi(v_{ii})$. By the first item, $\varphi(v_{ii}) = 1 \in G$. Thus, $\varphi(v_{ij}) = \varphi(v_{ji})^{-1}$.
	\end{proof}
	
	Thanks to the item $3$ of the last lemma, we can correctly define a mapping $\varphi$ on the set $V^T(\gamma)$. Since $\gamma(L) \subseteq \Gamma(L)$, and each vertex from $V(\Gamma)$ was constructed by a chain of vertices from $V(\gamma) \cup V^T(\gamma)$, then by item $2$, map $\varphi$ can be correctly extended to the set $V(\Gamma)$. As the map $\varphi$ on the edges of the graph $\gamma(L)$ was given through the set of its vertices $V(\gamma)$, $\varphi$ can therefore be defined in the same way on the edges of the graph $\Gamma(L)$. Thus, we have defined mapping $\varphi\colon \Gamma(L) \to G$. Next, we need an action of the graph $\Gamma(L)$ edges on the set $G\times G$.
	
	Let us assign the edge $e_{jk}^i\in E(\Gamma)$ to the operators $r_{jk}, \overline r_{jk}\colon G \times G \to G \times G$ acting by the rule $(x, y)r_{jk} = (x, \varphi(v_{jk})^{-1}y)$, and the edge $e^{jk}_i$ to the operators $l_{jk}, \overline l_{jk}\colon G \times G \to G \times G$ acting by the rule $(x, y)l_{jk} = (\varphi(v_{jk})x, y)$, $(x, y)\overline{l}_{ij} = (\varphi(v_{jk})^{-1}x, y)$. We will apply the operators from left to right, i.e., if $\chi_1$, $\chi_2$ are two operators defined above, then $(x, y)\chi_1\chi_2$ means the successive action first of the operator $\chi_1$, and then of $\chi_2$ on the pair $(x, y) \in G \times G$, i.e. $(x, y)(\chi_1\chi_2) = ((x, y)\chi_1)\chi_2$.
	
	Let us introduce one more notation: tilde '$\sim$' over the operators $\tilde r$, $\tilde l$ denotes either the presence of a bar or its absence. Now, we need some properties of the operators $\tilde r \in \{r_{jk}, \overline r_{jk}\}$, $\tilde l \in \{l_{jk}, \overline l_{jk}\}$.

	\begin{lemma}\label{lemma:opersProperties}
		The following properties for the operators $\tilde r_{ij}$, $\tilde l_{ij}$ holds:
		\begin{enumerate}
			\item $\tilde r_{ij}\tilde l_{kl} = \tilde l_{kl}\tilde r_{ij}$;
			
			\item $\overline{r}_{ij} = r_{ji}$, $\overline{l}_{ij} = l_{ji}$;
			
			\item $r_{ij}r_{jk} = r_{ik}$, $l_{ji}l_{kj}=l_{ki}$.
			
		\end{enumerate}
	\end{lemma}
	
	\begin{proof}
		
		1) Since the operator $\tilde r_{ij}$ acts on the right component of a pair from $G \times G$, and the operator $\tilde l_{ij}$ acts on the left one, they commute.  
		
		\noindent 2) By the definition of $\overline{r}_{ij}$
		\[(x, y)\overline{r}_{ij}=(x, \varphi(v_{ij})y) = (x, \varphi(v_{ji})^{-1}y) = (x, y)r_{ji}.\]
		
		\noindent 3) Consider the following expression
		\begin{align}
			(x, y)r_{ij}r_{jk} &= (x, \varphi(v_{jk})^{-1}\varphi(v_{ij})^{-1}y) \nonumber \\ 
			&= (x, \varphi(v_{kj})\varphi(v_{ji})y) \nonumber \\
			&= (x, \varphi(v_{ki})y) \nonumber \\ 
			&= (x, \varphi(v_{ik})^{-1}y) \nonumber \\ 
			&= (x, y)r_{ik}. \nonumber
		\end{align}
		
		\noindent The properties $2) - 3)$ for the operators $\overline{r}_{ij}$, $l_{ij}$, $\overline{l}_{ij}$ are proved analogously.
	\end{proof}
	
	Further, we will consider non-oriented paths in the graph $\Gamma(L)$. Let us fix the notation: if a path at some place goes along the edge $e_{jk}^i$ ($e^{jk}_i$ respectively), we will write $p = \ldots e_{jk}^i \ldots$ at the corresponding place ($p = \ldots e^{jk}_i \ldots$ respectively). If a path goes against the edge $e_{jk}^i$ ($e^{jk}_i$ respectively), we will write $p = \ldots \overline e_{jk}^i \ldots$ ($p = \ldots \overline e^{jk}_i \ldots$ respectively).

	Now let $p$ be an arbitrary non-oriented path in the graph $\Gamma(L)$. Define the action of the path $p$ on the set $G \times G$ by induction along the length of the path.
	Let $p$ be a path of length $1$. Define the value $(x, y)p$ by the following rule: 
	
	\[
	\begin{cases}
		(x,y)p = (x, y)r_{jk}, \text{ if } p = e_{jk}^i, \\
		(x,y)p = (x, y)l_{jk}, \text{ if } p = e^{jk}_i, \\
		(x,y)p = (x, y)\overline r_{jk}, \text{ if } p = \overline e_{jk}^i, \\
		(x,y)p = (x, y)\overline l_{jk}, \text{ if } p = \overline e^{jk}_i.
	\end{cases}
	\]
	
	Let $p = p_1p_2$, where $p_1$ is a path of nonzero length, and $p_2$ is a path of length $1$. Define the action of the path $p$ on pairs from $G \times G$ as the successive action by the paths $p_1$ and $p_2$, i.e., $(x, y)p = ((x, y)p_1)p_2$.
	Further, we identify a path with its action on $G \times G$. Now our goal is to show that two paths with the same start and end acts on the elements of $G \times G$ in a same way.
	
	\begin{lemma}\label{lemma:pathInvariancy}
		
		Let $v_{ij}$, $v_{kl}$ be two vertices in the graph $\Gamma(L)$, and let $p_1$, $p_2$ be two paths beginning at $v_{ij}$ and ending at $v_{kl}$. Then, $(x, y)p_1 = (x, y)p_2$ for any pair $(x, y) \in G \times G$.
	\end{lemma}
	
	\begin{proof}    
		By the definition of the path action on $G \times G$, they can be written as a composition of operators: 
		\[p_1 = \chi_1\chi_2\ldots \chi_{k_1}, \]
		\[p_2 = \xi_1\xi_2\ldots \xi_{k_2}, \]
		
		\noindent where $\chi_r, \xi_s \in \{r_{ij}, l_{ij}, \overline{r}_{ij}, \overline{l}_{ij}\}$. The operator $r_{ij}$ ($l_{ji}$ respectively) is written instead of $\chi_r$, $\xi_s$ if the path goes along the edge $e_{ij}^m$ ($e^{ji}_m$ respectively), and the operator $\overline{r}_{ji}$ ( $\overline{l}_{ij}$ respectively) is written if the path goes against the edge $e_{ij}^m$ ( $e^{ij}_m$ respectively).
		
		Now, we show that these two compositions of operators are equal. To do this, we will bring them to a certain normal form. By the first item of Lemma \ref{lemma:opersProperties}, one can first act by the operators $\tilde l$, and then by the operators $\tilde r$:
		\[p_1 = 
		\tilde l_{i_1i_2}\tilde l_{i_3i_4}\ldots \tilde l_{i_{n_1 - 1}i_{n_1}} 
		\tilde r_{j_1j_2}\tilde r_{j_3j_4}\ldots \tilde r_{j_{m_1 - 1}j_{m_1}}, 
		\]
		\[p_2 = 
		\tilde l_{s_1s_2}\tilde l_{s_3s_4}\ldots \tilde l_{s_{n_2 - 1}s_{n_2}} 
		\tilde r_{q_1q_2}\tilde r_{q_3q_4}\ldots \tilde r_{q_{m_2 - 1}q_2}. 
		\]
		
		\noindent Thus, it is necessary and sufficient to show that the following two equations hold:
		
		\begin{equation}\label{l:equation}
			\tilde l_{i_1i_2}\tilde l_{i_3i_4}\ldots \tilde l_{i_{n_1 - 1}i_{n_1}} = 
			\tilde l_{s_1s_2}\tilde l_{s_3s_4}\ldots \tilde l_{s_{n_2 - 1}s_{n_2}},
		\end{equation}
		\begin{equation}\label{r:equation}
			\tilde r_{j_1j_2}\tilde r_{j_3j_4}\ldots \tilde r_{j_{m_1 - 1}j_{m_1}} = 
			\tilde r_{q_1q_2}\tilde r_{q_3q_4}\ldots \tilde r_{q_{m_2 - 1}q_2}. 
		\end{equation}
		
		Consider the composition of operators $\tilde r_{j_1j_2}\tilde r_{j_3j_4}\ldots \tilde r_{j_{m_1 - 1}j_{m_1}}$ in the left-hand side of equation \eqref{r:equation}. If there is only one operator $r_{j_1,j_2}$ in this composition, the index $j_1 = j$ and $j_2 = l$, since we consider the path $p_1$ of length $1$ from $v_{ij}$ to $v_{kl}$ and $r_{j_1j_2} = r_{jl}$. If there is only one operator $\overline r_{j_1j_2}$, according to the same considerations, it follows that $j_1 = l$, $j_2 = j$, and $\overline r_{j_1j_2} = \overline r_{lj}$, which is equal to $r_{jl}$ by item 2 of Lemma \ref{lemma:opersProperties}.
		
		Further, suppose that there are at least two operators in the composition $\tilde r_{j_1j_2}\tilde r_{j_3j_4}\ldots \tilde r_{j_{m_1 - 1}j_{m_1}}$. Consider the composition of the first two operators $\tilde r_{j_1j_2}\tilde r_{j_3j_4}$. If the first operator has the form $r_{j_1j_2}$, then $j_1 = j$ since we consider the path $p_1$ with its beginning at the vertex $v_{ij}$. Consequently, $r_{j_1j_2} = r_{jj_2}$. If the first operator is equal to $\overline r_{j_2j_1}$, then $j_1 = j$ and $\overline r_{j_2j_1} = r_{jj_2}$ by the same consideration. Applying the same reasoning to the operator $\tilde r_{j_3j_4}$, we obtain $j_3 = j_2$, and hence $\tilde r_{j_3j_4} = r_{j_2j_4}$. Thus, we have obtained the equality $\tilde r_{j_1j_2}\tilde r_{j_3j_4} = r_{jj_2}r_{j_2j_4}$.
		
		Extending these arguments by induction and applying item $2$ of Lemma \ref{lemma:opersProperties} the required number of times, we obtain that the composition of operators $\tilde r_{j_1j_2}\tilde r_{j_3j_4}\ldots \tilde r_{j_{m_1 - 1}j_{m_1}}$ is equal to the composition of operators $r_{jj_1}r_{j_1j_2}\ldots r_{j_{m_1 - 1}l}$, where all neighboring indices between the operators are equal. Applying now item $3$ of Lemma \ref{lemma:opersProperties}, we obtain the following equation:
		\[r_{jj_1}r_{j_1j_2}\ldots r_{j_{m_1 - 1}l} = r_{jl}.\]

		\noindent It is analogous to the composition on the right-hand side of equation \eqref{r:equation}, we obtain 
		\[\tilde r_{q_1q_2}\tilde r_{q_3q_4}\ldots \tilde r_{q_{m_2 - 1}q_2} = r_{jl},\]
		as for the compositions of the operators $\tilde l$ in equation \eqref{l:equation}, the following equations are holds
		\[\tilde l_{i_1i_2}\tilde l_{i_3i_4}\ldots \tilde l_{i_{n_1 - 1}i_{n_1}} = l_{ki},\]
		\[\tilde l_{s_1s_2}\tilde l_{s_3s_4}\ldots \tilde l_{s_{n_2 - 1}s_{n_2}} = l_{ki}.\]
		
		Thus, equations \eqref{l:equation}, \eqref{r:equation} are satisfied, and two arbitrary non-oriented paths $p_1$ and $p_2$ from the vertex $v_{ij}$ to the vertex $v_{kl}$ are equal (as operators) to the operator $r_{jl}l_{ik}$. Consequently, $p_1$ and $p_2$ are equal as operators.
	\end{proof}
	
	It was shown in the last lemma that the action of any path $p$ from $v_{ij}$ to $v_{kl}$ is equal to that of the operator $r_{jl}l_{ki}$, which is defined by the following rule:
	\[(x, y)p = (x, y)r_{jl}l_{ki} = (\varphi(v_{ki})x, \varphi(v_{jl})^{-1}y).\]
	We will need this form of a path action in the next lemma.
	Now we can prove the main lemma of this section, stating that any graded matrix algebra with a natural basis can be embedded into an elementary graded algebra.
	
	\begin{lemma}\label{lemma:homgenousPut}
		
		Let $L \leqslant M_n(\mathbb{D})$ be a $G$-graded algebra with a natural basis; then it is identically embedded into an elementary $G$-graded algebra $M_n(\mathbb{D})$ as a graded algebra.
	\end{lemma}
	
	\begin{proof}
		
		To prove the lemma, we will construct an elementary grading on $M_n(\mathbb D)$ so that it contains $L$ as a graded subalgebra. To do so, it is necessary to choose a tuple of elements  $(g_1, \ldots, g_n) \in G^n$ such that, if $E_{ij} \in L_g$, then $g = g_ig^{-1}_j$.
		
		Let us consider the connected components in the graph $\Gamma(L)$. Since the graph is finite, the number of its connected components is finite as well. Introduce numbering from $1$ to $k$ on the set of connected components, where $k$ is the number of connected components, i.e., $\displaystyle\Gamma(L) = \bigsqcup_{i=1}^k \mathcal{O}_i$ and consider the first connected component $\mathcal{O}_1$. By Proposition \ref{prop:graphProps} there exists at least one vertex from $V(\gamma)$ in $\mathcal{O}_1$. Consider an arbitrary vertex $v_{ij} \in V(\gamma) \cap \mathcal{O}_1$ and the element $\varphi(v_{ij}) = g$. Let us begin fixing the elements defining the elementary grading. Fix the elements $g_i = 1$, $g_j = g^{-1}$ and consider the pair $(g_i, g_j) = (1, g^{-1}) \in G \times G$.
		
		Let $v_{kl} \in \mathcal{O}_1$. Consider an arbitrary non-oriented path $p$ from $v_{ij}$ to $v_{kl}$. Applying the path $p$ to the pair $(1, g^{-1})$, we obtain $(1, g^{-1})p$. Assign $g_k = \pi_1((1, g^{-1})p), g_l = \pi_2((1, g^{-1})p)$. By Lemma \ref{lemma:pathInvariancy}, the path action on elements $G \times G$ is invariant on the choice of the path $p$, but depends only on the initial and terminal vertices. Thus, this definition of $g_k$ and $g_l$ does not depend on the choice of the path $p$.
		
		We prove that $g_kg_l^{-1} = \varphi(v_{kl})$ by induction on the path length. If the path length is zero, the terminal vertex is $v_{ij}$. We fixed $g_i = 1$, $g_j = g^{-1} = \varphi(v_{ij})^{-1}$. So, $g_ig_j^{-1} = \varphi(v_{ij})$. Next, suppose that $g_{k_0}g_{l_0}^{-1} = \varphi(v_{k_0l_0})$ holds for all paths of length $m-1$ from the vertex $v_{ij}$ to the vertex $v_{k_0l_0}$. Choose a path $p$ of length $m$ from the vertex $v_{ij}$ to $v_{kl}$ and consider its subpath $p_0$ of length $m-1$ from $v_{ij}$ to $v_{k_0l_0}$. These two paths differ by one last edge. Thus, one of the following four cases is possible:
		
		\begin{enumerate}
			
			\item $p = p_0r_{l_0l}$: the terminal vertex is $v_{kl_0}$;
			\item $p = p_0\overline{r}_{ll_0}$: the terminal vertex is $v_{kl_0}$;
			\item $p = p_0l_{ll_0}$: the terminal vertex is $v_{k_0l}$;
			\item $p = p_0\overline{l}_{l_0l}$: the terminal vertex is $v_{k_0l}$.
		\end{enumerate}
		
		Consider the first case. By the induction hypothesis, $g_kg_{l_0}^{-1} = \varphi(v_{kl_0})$. Acting the pair $(g_k, g_{l_0})$ by the operator $r_{l_0l}$, we obtain
		
		\[(g_k, g_{l_0})r_{l_0l} = (g_k, \varphi(v_{l_0l})^{-1}g_{l_0}) \Rightarrow g_l = \varphi(v_{l_0l})^{-1}g_{l_0}.\]
		Consider the following chain of equations:
		\[g_kg_l^{-1} = g_kg_{l_0}^{-1}\varphi(v_{l_0l}) = \varphi(v_{kl_0})\varphi(v_{l_0l}) = \varphi(v_{kl}).\]
		
		\noindent The assertion in the first case is proved. The remaining cases are proved analogously.
		Continuing to do so, we can define the element $g_k$ by constructing a path to vertices of the form $v_{1k}, v_{2k}, \ldots, v_{nk}$ or to vertices of the form $v_{k1}, v_{k2}, \ldots, v_{kn}$ that lie in $\mathcal{O}_1$. Since a path from $v_{ij}$ to $v_{sk}$ acts by the operator $r_{jk}l_{si}$, the definition of $g_k$ does not depend on $s$. The same is true for vertices of the form $v_{ks}$. It remains to show that the definition of $g_k$ does not depend on the type of vertex $v_{ks_1}$ or $v_{s_2k}$. It is so whether we defined $g_k$ through a vertex of the form $v_{ks_1}$, and $h_k$ through a vertex of the form $v_{s_2k}$. So, by what was proved earlier, $g_kh_k^{-1} = \varphi(v_{kk}) = 1$, and hence $h_k = g_k$, so the definition of the elements $g_k$ is correct within one connected component.
		
		We proceed analogously with the remaining connected components. We will show that the correctness of defining the elements $g_k$ is not violated when considering the remaining connected components. By Proposition $\ref{prop:graphProps}$, if $v_{ij}, v_{ik} \in V(\Gamma)$, then they are in one connected component, just as $v_{jk}, v_{ik} \in V(\Gamma)$ are in one connected component. Thus, if we defined $g_k$ inside one connected component, we cannot define $g_k$ inside another connected component, since vertices of the form $v_{ks_1}$, $v_{s_2k}$ are in one connected component.

		We have constructed a set of elements $g_i \in G$, but maybe $g_k$ is not defined for some $i \in \{1, 2, \ldots , n\}$. But, in order to define an elementary grading, we need to fix a tuple $(g_1, \ldots , g_n)$ for all $i = 1, 2, \ldots, n$. So, if $g_k$ was not defined by the algorithm described above, define $g_k = 1 \in G$. Thus, we have defined the tuple $(g_1, \ldots , g_n)$ for each $i = 1, 2, \ldots , n$.
		
		Now let us consider the algebra $M_n(\mathbb{D})$ elementary graded by this tuple, $M_n(\mathbb D) = R = \bigoplus_{g\in G}R_g$. Recall that since $L$ is an algebra with a natural basis, which means by definition that its basis is a set of matrix units. Define the homomorphism $\tau\colon L \to R$ on the basis elements: $\tau\colon E_{ij} \mapsto E_{ij}$. It is obvious that $\tau$ is an injective.
		
		Let $E_{ij} \in L_g$. Consequently, $g = \varphi(v_{ij}) = g_ig^{-1}_j$, and hence, the matrix unit $E_{ij} \in R_g$ in the elementary graded algebra $R = M_n(\mathbb{D}) = \bigoplus_{g \in G}R_g$. Thus, $\tau(L_g) \subseteq R_g$. We have obtained that $\tau$ is a homomorphism of $G$-graded algebras, and since it is injective, it is an embedding of graded algebras.
	\end{proof}
	
	The previous lemma demonstrates an algorithm for constructing the tuple $\overline g = (g_1, \ldots , g_n)$. However, this tuple is not the only one that the elementary graded algebra with respect to the tuple $\overline g$ contains $L$ as a graded subalgebra. By Proposition \ref{prop:shift}, the tuple $(g_1h, \ldots, g_nh)$ defines the same elementary grading on $M_n(\mathbb D)$ for an arbitrary element $h \in G$.
	
	Recall that at the beginning of the algorithm, the initial vertex $v_{ij} \in \gamma(L)$ was considered, and it was assumed that $g_i = 1$, $g_j = g^{-1}$, where $E_{ij} \in L_g$. Accordingly, in order to obtain the tuple $(g_1h, \ldots, g_nh)$ as the result of the algorithm to be one, it should initially be assumed that $g_i = h$, $g_j = (gh)^{-1}$.

	\section{Embedding in a general case}\label{sect:commonCase}
	
	Now our aim is to show that an arbitrary graded matrix algebra over a division ring is embedded into an elementary graded algebra. The embedding will be constructed into an algebra graded by a quotient group of the group $G$.
	
	Let us recall that the classical homomorphism of graded algebras mentioned in Section \ref{Prel} is a map between algebras graded by the same group. In order to perform embedding from a graded algebra into an algebra graded by another group, we need the following definition.
	
	\begin{definition}\label{def:gradedHom}
		
		Let $G$ and $H$ be groups, $A = \displaystyle\bigoplus_{g\in G}A_g$, and $B = \displaystyle\bigoplus_{h \in H}B_h$ be algebras graded by $G$ and $H$, respectively. We call a homomorphism of algebras $\varphi\colon A \to B$ a homomorphism of a G-algebra into an H-algebra if
		
		\begin{enumerate}
			
			\item For each $g \in G$ there exists $h\in H$ such that $\varphi (A_g) \subseteq B_h$;
			\item If $\varphi (A_{g_1}) \subseteq B_{h_1}$, $\varphi (A_{g_2}) \subseteq B_{h_2}$, then $\varphi (A_{g_1g_2}) \subseteq B_{h_1h_2}$.			
		\end{enumerate}
	\end{definition}
	
	\begin{proposition}
		
		Definition \ref{def:gradedHom} generalizes the classical definition of the homomorphism of graded algebras, i.e., the homomorphism of $G$-graded algebras $A$ and $B$ satisfies Definition \ref{def:gradedHom}.
	\end{proposition}
	
	\begin{proof}
		
		Let $\displaystyle A = \bigoplus_{g \in G}A_g, B = \bigoplus_{g \in G}B_g$ and let $\varphi\colon A \to B$ be the homomorphism such that $\varphi(A_g) \subseteq B_g$ for all $g \in G$. It is obvious that the first item of definition \ref{def:gradedHom} is satisfied. Moreover, we have $\varphi(A_{g_1})\subseteq B_{g_1}$, $\varphi(A_{g_2})\subseteq B_{g_2}$, and $\varphi(A_{g_1g_2})\subseteq B_{g_1g_2}$ for an arbitrary pair of elements $g_1, g_2 \in G$, which corresponds to the second item of definition \ref{def:gradedHom}.
	\end{proof}
	
	To construct the required quotient group, we need to introduce some more definitions. The support of a matrix $A = (a_{ij}) \in M_n(\mathbb D)$ is the set of pairs $suppA = \{ (i, j) \ | \ a_{ij} \neq 0\} \subseteq \mathbb{N}^2_+$, where $\mathbb N_+$ denotes the positive natural numbers. The support of nonempty set of matrices $S \subseteq M_n(\mathbb D)$ is the union of the supports of all its matrices:
	\[suppS = \bigcup_{A \in S} suppA.\]

	\begin{example}
		
		The support of the matrix unit $E_{ij} \in M_n(\mathbb D)$ is equal to $\{(i, j)\}$.
	\end{example}
	
	\begin{example}
		
		Let $A = \begin{pmatrix} 3 & 5 \\ 0 & 0 \end{pmatrix} \in M_2(\mathbb{R})$, then $supp(A) = \{(1, 1), (1, 2)\}$.
	\end{example}
	
	Also, we will need the product of pairs $\cdot : \mathbb{N}^2_+ \to \{0\} \cup \mathbb{N}^2_+$, which corresponds to the product of matrix units. Define it as follows:

	\begin{equation}
		\begin{cases}
			(i,j)\cdot (k,l) = (i, l)\text{, if } j = k,\\
			(i,j)\cdot (k,l) = 0\text{, otherwise.}
		\end{cases}
	\end{equation}
	
	Remark. The symbol zero '$0$' in the latter definition should be understood as a formal symbol introduced only so as to define the product of any pair. 
	
	Let us define the product of supports for two matrices $A, B \in M_n(\mathbb D)$ by the following rule:
	\[suppA\cdot suppB = \{(i, j) \cdot (k, l) \ | \ (i, j) \in suppA, (k, l) \in suppB, (i, j)\cdot (k, l) \neq 0\}.\]
	
	Further, the product of the supports of two matrix $S_1, S_2 \subseteq M_n(\mathbb{D})$ sets is the union of all pairwise products of the matrix supports in the given sets 
	\[suppS_1\cdot suppS_2 = \bigcup_{A\in S_1, B \in S_2} suppA\cdot suppB.\]
	The product of the matrix support $A \in M_n(\mathbb{D})$ and the set support $S \subseteq M_n(\mathbb D)$ is the product of the supports of the sets $\{A\}$ and $S$.
	
	Let us establish some properties of matrix supports and their products.
	
	\begin{proposition}\label{prop:suppSubset}
		
		The support of the matrix product is embedded into the product of their supports, i.e., $suppAB \subseteq suppA\cdot suppB$.
	\end{proposition}
	
	\begin{proof} 
		
		Let $A, B \in M_n(\mathbb D)$ and $(i, j) \in suppAB$; then, by the definition of support,
		\begin{align}
			0 \neq E_{ii}ABE_{jj} &= E_{ii}(\sum a_{kl}E_{kl})(\sum b_{rs}E_{rs})E_{jj} \nonumber\\
			&=(\sum a_{il}E_{il})(\sum b_{rj}E_{rj}) \nonumber\\ 
			&= (\sum a_{il}b_{lj})E_{ij} \neq 0. \nonumber
		\end{align}
		Due to the fact that the latter expression is not equal to zero, there are elements $a_{il_1}$ and $b_{l_1j}$ which differ from zero. Consequently, $(i, l_1) \in suppA$, $(l_1, j) \in suppB$ and $suppA \cdot suppB \ni (i, l_1) \cdot (l_1, j) = (i, j).$.
		
		The analogous statement is true if $S_1$ and $S_2$ are sets of matrices. In this case, embedding $supp(S_1S_2) \subseteq suppS_1 \cdot suppS_2$ is hold.
	\end{proof}
	
	The following example shows that the product of the matrix supports is not necessarily equal to the support of the matrix product.
	
	\begin{example}\label{example:suppCounterexample}
		
		Let $A = E_{11} - E_{12}$, $B = E_{11} + E_{21} \in M_2(\mathbb D)$; then $suppA = \{(1, 1), (1, 2)\}$, $suppB = \{(1, 1), (2, 1)\}$. Using direct computations, one can check that $suppA\cdot suppB = \{(1, 1)\}$. On the other hand, $AB = 0$. Consequently, $suppAB = \emptyset$.
	\end{example}
	
	Further, we will need a certain inverse of Proposition \ref{prop:suppSubset} for sets of matrices consisting of matrix units.
	
	\begin{proposition}\label{prop:multSuppUnit}
		
		Let $S_1$ and $S_2$ be sets consisting of matrix units; then
		\begin{enumerate}
			
			\item $S_1S_2$ is a set consisting of matrix units and, possibly, the zero matrix;
			\item $S_1 \subseteq S_2$ if and only if $suppS_1 \subseteq suppS_2$;
			\item $supp(S_1S_2) = suppS_1\cdot suppS_2$.
		\end{enumerate}
	\end{proposition}
	
	\begin{proof} 
		
		1) It is obvious.
		
		\noindent 2) Let us prove the necessity. Let $(i, j) \in suppS_1$. Since $S_1$ consists of matrix units, then $E_{ij} \in S_1$. Accordingly, $E_{ij} \in S_2$ and $(i,j) \in suppS_2$.
		
		Let us prove the sufficiency. Let $E_{ij} \in S_1$, then $(i,j) \in suppS_1$. Consequently, $(i, j) \in suppS_2$, and since $S_2$ consists of matrix units, it follows that $E_{ij} \in S_2$.
		
		\noindent 3) Embedding $supp(S_1S_2) \subseteq suppS_1\cdot suppS_2$ is a direct consequence of Proposition \ref{prop:suppSubset}. Let us prove inverse embedding. Let $(i,k) \in suppS_1 \cdot suppS_2$, i.e. $(i, j) \in suppS_1$, $(j, k) \in suppS_2$ for some $j$. Consequently, there are two matrix units $E_{ij} \in S_1$, $E_{jk} \in S_2$, therefore, $E_{ik} = E_{ij}E_{jk} \in S_1S_2$ and $(i, k) = (i, j)\cdot (j, k) \in supp(S_1S_2)$.
	\end{proof}
	
	\begin{definition}
		
		Consider an arbitrary $G$-graded matrix algebra $\displaystyle L = \bigoplus_{g \in G} L_g$. The support group $Supp_L(G)$ of the algebra $L$ will be called the normal closure in the group $G$ of the following set:
		\[\{g_1\cdot\ldots\cdot g_n  (h_1\cdot\ldots\cdot h_m)^{-1} \ | \ suppL_{g_1} \cdot\ldots\cdot suppL_{g_n} \cap suppL_{h_1} \cdot\ldots\cdot supp L_{h_m} \neq \emptyset \}.\]
	\end{definition}
	
	The following theorem shows that if $L$ is a $G$-graded matrix algebra over a division ring $\mathbb{D}$, it can be identically embedded into an elementary $G/Supp_L(G)$-graded algebra.

	\begin{theor}\label{theor:main}
		
		Let $\displaystyle L = \bigoplus_{g\in G} L_g$ be a $G$-graded matrix algebra over a division ring $\mathbb{D}$, $H = G/Supp_L(G)$; then the algebra $L$ is identically embedded into an elementary $H$-graded algebra as a graded algebra in the sense of definition \ref{def:gradedHom}.
	\end{theor}

	\begin{proof} 
		
		Define a series of sets for each $h \in H$:
		\[M_h = \{E_{ij} \ | \ \exists g_1, \ldots, g_n \in G : \prod g_i \in h; (i,j) \in \prod suppL_{g_i}\}.\]
		
		\noindent We assign the linear span over $\mathbb{D}$ to each set $M_h$:
		\[R_h = span\{M_h\}.\]
		
		\noindent and consider the $\mathbb D$-module $\displaystyle R = \sum_{h\in H}R_h$. We prove that this sum is direct.
		
		If the sum is not direct, there exists an element $A \in R$ such that it has a non-unique decomposition by the elements from the submodules $R_h, h \in H$. Consequently, $\displaystyle A = \sum_{h \in H} A_h = \sum_{h \in H} B_h$ and $\displaystyle \sum_{h \in H}(A_h - B_h) = 0$. By assumption, there exists an element $h_0 \in H$ such that

		\begin{equation}\label{directEq1}
			0 \neq A_{h_0} - B_{h_0} = \sum_{h \neq h_0}(B_h - A_h). 
		\end{equation}
		
		\noindent Multiply expression \eqref{directEq1} by all diagonal matrix units from $M_n(\mathbb D)$ on the left and on the right.
		
		\begin{equation}\label{directEq2}
			E_{ii}(A_{h_0} - B_{h_0})E_{jj} = \sum_{h \neq h_0}E_{ii}(B_h - A_h)E_{jj}. 
		\end{equation}
		
		\noindent If all the expressions from \eqref{directEq2} are equal to zero, then $A_{h_0} - B_{h_0} = 0$. It contradicts \eqref{directEq1}. Consequently, there exist $i$ and $j$ such that \eqref{directEq2} is not equal to zero. Thus, for some $i, j$, the following holds:
		
		\begin{equation}\label{directEq3}
			0 \neq E_{ii}(A_{h_0} - B_{h_0})E_{jj} = \sum_{h \neq h_0}E_{ii}(B_h - A_h)E_{jj}.  
		\end{equation}
		
		\noindent It follows from \eqref{directEq3} that $(i, j) \in supp R_{h_0}$ and $(i, j) \in \bigcup_{h \neq h_0}suppR_h$. Taking the latter into account as well as the fact that each module $R_h$ is the linear span of certain matrix units, we conclude that $E_{ij}$ lies simultaneously in $R_{h_0}$ and in $\displaystyle \sum_{h \neq h_0}R_h$. Consequently, there exists an element $h_1 \in H\backslash \{h_0\}$ such that $E_{ij} \in R_{h_0}, R_{h_1}$. Hence, by the definition of $R_{h_0}$ and $R_{h_1}$, there are two sequences of elements $g_{0, 1}, g_{0, 2}, \ldots, g_{0, n},\  g_{1, 1}, g_{1, 2}, \ldots, g_{1, m} \in G$ such that
		
		\begin{enumerate}
			\item $\displaystyle \prod_{k = 1}^n g_{0, k} \in h_0$, $\displaystyle \prod_{k = 1}^m g_{1, k} \in h_1$,
			\item $\displaystyle (i, j) \in \prod_{k = 1}^n supp L_{g_{0, k}}$, $\displaystyle (i, j) \in \prod_{k = 1}^m supp L_{g_{1, k}}$.
		\end{enumerate}
		
		\noindent It follows from the second item that $\prod supp L_{g_{0, k}} \cap \prod supp L_{g_{1, k}} \neq \emptyset$. Consequently, $\prod g_{0, k}(\prod g_{1, k})^{-1} \in Supp_L(G)$, and the element $\prod g_{0, k}(\prod g_{1, k})^{-1}$ represents the coset of the unit in $H$. Consequently, $h_0 = h_1$ in the group $H$. However, we had the element $h_1 \in H \backslash \{h_0\}$. So, there is a contradiction, and hence the sum $\displaystyle R = \bigoplus_{h\in H}R_h$ is direct.
		
		We prove that embedding $R_h R_s \subseteq R_{hs}$ for arbitrary $h, s \in H$ is true. First, it is necessary to show that $R_h R_s \subseteq R_{hs} \Leftrightarrow M_hM_s \subseteq M_{hs}$. Necessity: since $R_h R_s \subseteq R_{hs}$, then $M_hM_s \subseteq R_{hs}$. By item 1 of proposition $\ref{prop:multSuppUnit}$, $M_hM_s$ is a set consisting only of matrix units, but $M_{hs}$ is the maximal subset of $R_{hs}$ consisting of matrix units. Consequently, $M_hM_s \subseteq M_{hs}$. Sufficiency is obvious, since $R_h, R_s$, and $R_{hs}$ are the linear spans of $M_h, M_s$, and $M_{hs}$, respectively.  
		
		According to item 2 of proposition \ref{prop:multSuppUnit}, it is true that $M_hM_s \subseteq M_{hs} \Leftrightarrow supp(M_hM_s) \subseteq supp(M_{hs})$, and by item 3 of the same proposition $supp(M_hM_s) \subseteq supp(M_{hs}) \Leftrightarrow supp(M_h)\cdot supp(M_s) \subseteq supp(M_{hs})$. Thus, the following chain of equivalences is true:
		\[R_h R_s \subseteq R_{hs} \Leftrightarrow M_hM_s \subseteq M_{hs} \Leftrightarrow supp(M_hM_s) \subseteq suppM_{hs} \Leftrightarrow suppM_h\cdot suppM_s \subseteq suppM_{hs}.\]
		
		\noindent I.e., to prove embedding $R_hR_s \subseteq R_{hs}$, it is necessary to prove embedding supports $suppM_h\cdot suppM_s \subseteq suppM_{hs}$. By the definition of multiplication of supports, the following equality holds:
		\[suppM_h\cdot suppM_s = \{(i,k)|\exists g_1, \ldots, g_n, \ t_1, \ldots, t_m \in G: \prod g_k \in h, \prod t_k \in s, \]
		\[\hspace{15em} \exists (i,j) \in \prod suppL_{g_k}, (j, k) \in \prod suppL_{t_k}\}.\]
		
		Let us relabel the elements of this set: $r_i = g_i$, $i = 1, \ldots, n$, and $r_{n+i} = t_i$, $i = 1, \ldots, m$. Then this set will be written as follows:
		\[suppM_h\cdot suppM_s = \{(i, k) | \exists r_1, \ldots r_{n+m} \in G: \prod r_i \in hs, (i, k) \in \prod suppL_{r_i}\}.\]
		
		\noindent It is easy to see that the latter set is a subset of $suppM_{hs}$. Thus, the inclusion $R_hR_s \subseteq R_{hs}$ has been established for arbitrary $h, s \in H$. 
		
		Also, it follows from $R_hR_s \subseteq R_{hs}$ that $R$ is closed under multiplication. It is enough to consider two elements $a, b \in R$, decompose them into projections $a = \sum_{h \in H} a_h$, $b = \sum_{s \in H} b_s$, where $a_h, b_h \in R_h$, and multiply: $ab = (\sum_{h \in H} a_h)(\sum_{s \in H} b_s) = \sum_{h,s\in H}a_hb_s$. By what has been proved, all products $a_hb_s$ lie in the homogeneous components $R_{hs}$. Consequently, the product of the elements $a$ and $b$ lies in the sum of the homogeneous components, i.e., lies in $R$.
		
		Thus, $R$ is an $H$-graded algebra over the division ring $\mathbb D$. By construction, the algebra $R$ is the linear span of certain matrix units, with all matrix units in $R$ being homogeneous, i.e., $R$ is an $H$-graded algebra with a natural basis. Consequently, by Lemma \ref{lemma:homgenousPut}, the algebra $R$ is identically embedded into an elementary $H$-graded algebra $C = M_n(\mathbb D)$ as a graded algebra, i.e., there exists an injective homomorphism of H-graded algebras:
		\[\xi: R \to C; \hspace{2em} \xi: a \mapsto a.\]
		
		It should be noted that $L$ is identically embedded into $R$ in the sense of Definition \ref{def:gradedHom}. It is clear that $L \subseteq R$ by construction. It means that the identity homomorphism is defined:
		\[\tau: L \to R; \hspace{2em} \tau: a \mapsto a.\]
		
		We prove that $\tau$ satisfies Definition \ref{def:gradedHom}. Consider the homogeneous component $L_g$, $g \in G$. Let $h \in H$ be a coset of the element $g$. By the definition of $M_h$, the inclusion $suppL_g \subseteq suppM_h$ holds. Consequently, $L_g \subseteq R_h$, or, in the terms of introduced homomorphism, $\tau( L_g) \subseteq R_h$. Let $g_1, g_2 \in G$ and $h_1, h_2 \in H$ be the cosets of the elements $g_1$ and $g_2$, respectively. Then, as it was proved above, $\tau (L_{g_1}) \subseteq R_{h_1}$, $\tau (L_{g_2}) \subseteq R_{h_2}$. Since $h_1h_2$ is the coset of the element $g_1g_2$, embedding $\tau (L_{g_1g_2}) \subseteq R_{h_1h_2}$ is valid. Thus, $\tau$ satisfies definition \ref{def:gradedHom}, and $L$ is identically embedded into $R$ as a graded algebra. So, the homomorphism $\xi\circ\tau$ embeds the algebra $L$ into the elementary $H$-graded algebra $C$.
	\end{proof}
	
	Let us establish some properties of this embedding. To begin with, consider a certain counterexample which would seem to nullify the results of the previous theorem. It is obvious that any graded matrix subalgebra is identically embedded into a trivially graded matrix algebra in the sense of \ref{def:gradedHom}.
	
	\begin{example}
		
		Consider a $G$-graded subalgebra $L \leqslant M_n(\mathbb{D})$. It is obvious that it is embedded into the $G/G$-graded algebra $M_n(\mathbb{D})$, which is elementary graded, since the group $G/G$ is trivial.
	\end{example}
	
	Let us show that the embedding constructed in the previous theorem is somehow the best possible one.
	
	\begin{lemma}
		
		Let $\displaystyle L = \bigoplus_{g\in G} L_g$ be a $G$-graded matrix algebra over a division ring $\mathbb{D}$. The subgroup $Supp_L(G) \triangleleft G$ is the minimal one such that $L$ is embedded identically into an elementary $G/Supp_L(G)$-graded algebra.
	\end{lemma}
	
	\begin{proof}
		
		Consider the set $Norm(G)$ of normal subgroups in $G$:
		\[Norm(G) = \{K \triangleleft G\}.\]
		
		It is evident that the given set is partially ordered by inclusion. Moreover, it is a lattice. This lattice contains the support group $Supp_L(G) \in Norm(G)$. Consider the upper ideal $I$ of the element $Supp_L(G)$ in this lattice
		\[I = \{K \in Norm(G) \ | \ Supp_L(G) \subseteq K\}.\]
		
		It is obvious that all elements of this ideal are such subgroups $K$ that $L$ is identically embedded into a $G/K$-elementary graded algebra. While proving Theorem \ref{theor:main}, one only has to replace $Supp_L(G)$ by $K$ such that $Supp_L(G) \subseteq K$, $K \triangleleft G$. At the same time, $Supp_L(G)$ is minimal among all such groups from $I$.
		
		Now let $K \in Norm(G)$, $K \not\in I$. Then in $G$ there exist elements $g_1, \ldots g_n$ and $h_1, \ldots , h_m$ such that $\prod_{k = 1}^n suppL_{g_k} \cap \prod_{s = 1}^m supp L_{h_s} \neq \emptyset$ and $\prod_{k = 1}^ng_k(\prod_{s = 1}^mh_s)^{-1} \not\in K$. Assume the opposite. Consider all sequences of elements $g_k$, $h_s \in G$ such that $\prod_{k = 1}^n suppL_{g_k} \cap \prod_{s = 1}^m supp L_{h_s} \neq \emptyset$. By assumption, all elements $(\prod_{k = 1}^ng_k)(\prod_{s = 1}^mh_s)^{-1} \in K$, and hence the group $K$ contains the subgroup $Supp_L(G)$, because the group $Supp_L(G)$ is minimal containing all elements $(\prod_{k = 1}^ng_k)(\prod_{s = 1}^mh_s)^{-1}$. It contradicts the fact that $K \not\in I$.
		
		Since $\prod suppL_{g_k} \cap \prod supp L_{h_s} \neq \emptyset$, let us fix  the pair from the intersection $(i, j) \in \prod suppL_{g_k} \cap \prod supp L_{h_s}$ and the group $H = G/K$. Suppose that the algebra $L$ is identically embedded into an elementary $H$-graded algebra$C = M_n(\mathbb D)$. As $(i, j) \in \prod suppL_{g_k}$, there exist pairs $(l_0, l_1), (l_1, l_2), \ldots, (l_{n-1}, l_n)$, where $l_0 = i$, $l_n = j$, such that $(l_{k-1}, l_k) \in supp L_{g_k}$. Since we are considering an identical embedding and in the elementary graded algebra $C$, all matrix units are homogeneous and the following holds: $E_{l_{k-1}l_k} \in C_{\overline {g_k}}$, and hence $E_{ij} = \prod_{k=1}^n E_{l_{k-1}l_k} \in C_{\overline{g_1\cdot\ldots\cdot g_n}}$. Carrying out analogous considerations for the elements $h_s$, we obtain that $E_{ij} \in C_{\overline {h_1\cdot\ldots\cdot h_m}}$. Consequently, $C_{\overline{g_1\cdot\ldots\cdot g_n}} = C_{\overline {h_1\cdot\ldots\cdot h_m}}$ and $\overline{g_1\cdot\ldots\cdot g_n} =\overline {h_1\cdot\ldots\cdot h_m}$ in the group $H$, which is equivalent to $(g_1\cdot\ldots\cdot g_n)(h_1\cdot\ldots\cdot h_m)^{-1} \in K$. There is a contradiction.
	\end{proof}
	
	Denote the elements of the group $\mathbb Z_2 \times Z_2$ by $\{1,a,b,c\}$, where $1$ is the neutral element in the group, but all the others are of order $2$. Paper \cite{BahSeg} presents an example of the grading on the algebra $M_2(\mathbb F)$, where $char \mathbb F \neq 2$, and it is not elementary. Let us recall it:
	
	\begin{example}[{\cite{BahSeg}}]\label{example:kleinM2}
		
		Let $\mathbb{F}$ be a field, $char \mathbb{F} \neq 2$. Consider the $\mathbb Z_2\times\mathbb Z_2$-grading on the algebra $L = M_2(\mathbb{F})$, where the components $L_1, L_a, L_b, L_c$ are the linear spans of the matrices
		\[
		\begin{pmatrix} 
			1 & 0 \\ 
			0 & 1 
		\end{pmatrix},
		\begin{pmatrix} 
			1 & 0 \\ 
			0 & -1 
		\end{pmatrix},
		\begin{pmatrix} 
			0 & 1 \\ 
			1 & 0 
		\end{pmatrix},
		\begin{pmatrix} 
			0 & 1 \\ 
			-1 & 0 
		\end{pmatrix}
		\]
		respectively. Since $L_{1} = Z(L)$, $L$ is not elementary graded in accordance with Corollary \ref{coro:badCenter}.
		
	\end{example}
	
	Let us show how to embed this algebra into an elementary graded algebra, where the grading is assigned by a certain factor group of the group $\mathbb Z_2\times\mathbb Z_2$.
	
	\begin{example}
		
		Let $L$ be the algebra from Example \ref{example:kleinM2}. Find $suppL_g$ for all $g \in \mathbb Z_2 \times\mathbb Z_2$.
		
		\begin{enumerate}
			\item[] $suppL_{1} = suppL_{a} = \{(1,1), (2,2)\}$;
			\item[] $suppL_{b} = suppL_{c} = \{(1,2), (2,1)\}$.
		\end{enumerate}
		
		Next, it is checked by direct computations that $suppL_g \cdot suppL_h = supp L_{gh}$ for all $g, h \in \mathbb Z_2 \times\mathbb Z_2$. Thus, any product of a finite number of supports $suppL_{g_1} \cdot \ldots \cdot suppL_{g_n}$ will be equal to $suppL_g$, where $g = g_1\cdot\ldots\cdot g_n \in \mathbb Z_2 \times\mathbb Z_2$.
		Obviously, a non-empty intersection occurs only in two cases: $suppL_{1} \cap suppL_{a} \neq \emptyset$ and $suppL_{b} \cap suppL_{c} \neq \emptyset$. Thus, $Supp_L(\mathbb Z_2 \times\mathbb Z_2)$ is the normal closure of the following set:
		\[\{1\cdot a^{-1}, b \cdot c^{-1}\}\]
		\[Supp_L(\mathbb Z_2 \times\mathbb Z_2) = \{1, a\}\]
		
		Consider the group $H = (\mathbb Z_2 \times\mathbb Z_2)/Supp_L(\mathbb Z_2 \times\mathbb Z_2)$. it is established by direct computations that $H = \{\overline{1}, \overline{b}\} \cong \mathbb Z_2$, where $\overline{1} = \{1, a\}$, $\overline{b} = \{b, c\}$.
		
		Consider the elementary $H$-graded algebra $R = M_2(\mathbb{F})$ (we will establish the specific grading later). It follows from the previous paragraph that the embedding $\tau\colon L \to R$ constructed in Theorem \ref{theor:main} will map the homogeneous components according to the following rule:
		\[\tau(L_{1}), \tau(L_{a}) \subseteq R_{\overline{1}}, \hspace{2em} \tau(L_{b}), \tau(L_{c}) \subseteq R_{\overline{b}}.\]
		
		\noindent Thus, it is not difficult to determine in what homogeneous components of the elementary $H$-graded algebra $R$ the matrix units lie: $E_{11}, E_{22} \in R_{\overline{1}}$, $E_{12}, E_{21} \in R_{\overline{b}}$. Consequently, the algebra $R$ is graded by a tuple $(g_1, g_2) \in H \times H$ such that $g_1\cdot g_1^{-1} = g_2\cdot g_2^{-1} = \overline{1}$, $g_1\cdot g_2^{-1} = g_2 \cdot g_1^{-1} = \overline{b}$. Obviously, there are only two such tuples in $H\times H$. The first is $(\overline{1}, \overline{b})$, the second is $(\overline{b}, \overline{1})$. Note that the second tuple can be obtained from the first one by shifting on the element $\overline b \in H$ in accordance with Proposition \ref{prop:shift}.

	\end{example}
	
	\bigskip
	
	\section*{Acknowledgments}
	
	The work was supported by the Foundation for the Development of Theoretical Physics and Mathematics «BASIS» 23-7-2-14-1.

	\medskip
	\noindent Pavel Sokolov \\
	Novosibirsk state university; \\
	email: p.sokolov@g.nsu.ru
	
\end{document}